\documentclass[reqno,centertags,12pt]{amsart}
\usepackage{amsmath,amsthm,amscd,amssymb}
\usepackage{latexsym}
\usepackage{graphicx}
\usepackage{enumitem}
\usepackage[mathscr]{eucal}

\usepackage{hyperref}

\newcommand{\bbC}{{\mathbb{C}}}
\newcommand{\bbD}{{\mathbb{D}}}

\newcommand{\bbN}{{\mathbb{N}}}

\newcommand{\bbR}{{\mathbb{R}}}

\newcommand{\bbZ}{{\mathbb{Z}}}

\newcommand{\fre}{{\mathfrak{e}}}
\newcommand{\frf}{{\mathfrak{f}}}
\newcommand{\frg}{{\mathfrak{g}}}

\newcommand{\cz}{{\mathbf{z}}}

\newcommand{\al}{\alpha}
\newcommand{\be}{\beta}

\newcommand{\calC}{{\mathcal{C}}}

\newcommand{\calG}{{\mathcal G}}

\newcommand{\dott}{\,\cdot\,}

\newcommand{\lb}{\label}

\newcommand{\ol}{\overline}

\newcommand{\wti}{\widetilde  }

\newcommand{\dist}{\text{\rm{dist}}}

\newcommand{\bi}{\bibitem}

\newcommand{\beq}{\begin{equation}}
\newcommand{\eeq}{\end{equation}}
\newcommand{\ba}{\begin{align}}
\newcommand{\ea}{\end{align}}

\renewcommand{\Im}{\operatorname{Im}}

\newcommand{\bs}{\backslash}
\newcommand{\hatC}{\bbC\cup\{\infty\}}
\newcommand{\hatR}{\bbR\cup\{\infty\}}
\newcommand{\pd}{\partial}
\DeclareMathOperator{\PW}{\mathcal{PW}}
\newcommand{\sgn}{\mathrm{sgn}}
\newcommand{\eps}{\varepsilon}

\newcounter{smalllist}
\newenvironment{SL}{\begin{list}{{\rm\roman{smalllist})}}{%
\setlength{\topsep}{0mm}\setlength{\parsep}{0mm}\setlength{\itemsep}{0mm}%
\setlength{\labelwidth}{2em}\setlength{\leftmargin}{2em}\usecounter{smalllist}%
}}{\end{list}}

\newcommand{\comm}[1]{}

\allowdisplaybreaks
\numberwithin{equation}{section}

\newtheorem{theorem}{Theorem}[section]
\newtheorem{proposition}[theorem]{Proposition}

\newtheorem{corollary}[theorem]{Corollary}
\theoremstyle{definition}
\newtheorem{example}[theorem]{Example}

\newtheorem*{remark}{Remark}
\newtheorem*{remarks}{Remarks}
\newtheorem*{definition}{Definition}

\newcommand{\norm}[1]{\lVert#1\rVert}

\begin{document}

\title[Residual Polynomials]{Asymptotics of Chebyshev Polynomials, V. Residual Polynomials}
\author[J.~S.~Christiansen, B.~Simon and M.~Zinchenko]{Jacob S.~Christiansen$^{1,4}$, Barry Simon$^{2,5}$ \\and
Maxim~Zinchenko$^{3,6}$}

\thanks{$^1$ Centre for Mathematical Sciences, Lund University, Box 118, 22100 Lund, Sweden.
 E-mail: stordal@maths.lth.se}

\thanks{$^2$ Departments of Mathematics and Physics, Mathematics 253-37, California Institute of Technology, Pasadena, CA 91125.
E-mail: bsimon@caltech.edu}

\thanks{$^3$ Department of Mathematics and Statistics, University of New Mexico,
Albuquerque, NM 87131, USA; E-mail: maxim@math.unm.edu}

\thanks{$^4$ Research supported by VR grant 2018-03500 from the Swedish Research Council.}

\thanks{$^5$ Research supported by NSF grant DMS-1665526.}

\thanks{$^6$ Research supported in part by Simons Foundation grant CGM-581256.}

\thanks{$^7$ Dedicated with great respect to the memory of Richard Askey, 1933--2019.}

\

\date{\today}

\begin{abstract} We study residual polynomials, $R_{x_0,n}^{(\fre)}$, $\fre\subset\bbR$, $x_0\in\bbR\bs\fre$, which are the degree at most $n$ polynomials with $R(x_0)=1$ that minimize the $\sup$ norm on $\fre$.  New are upper bounds on their norms (that are optimal in some cases) and Szeg\H{o}--Widom asymptotics under fairly general circumstances.  We also discuss several illuminating examples and some results in the complex case.
\end{abstract}

\maketitle

\section{Introduction} \lb{s1}

Dick Askey was a great fan of Gabor Szeg\H{o} as seen by his wonderful, readable notes on Szeg\H{o}'s papers in Szeg\H{o}'s complete works \cite{Askey}, which Dick edited as a clear labor of love.  In particular, it is clear that Askey was fond of Szeg\H{o} asymptotics so we are pleased to be able to dedicate this paper on an extension of such asymptotics to Dick's memory.  Even though Askey's work was largely on the algebraic side of the theory of orthogonal (and other) polynomials while our own has mainly been on the analytic side, his work has been so deep and so broad that it has impacted us.  In addition, Dick's warmth and kindness are legion.

Let $\fre \subset \bbC$ be a compact, not finite, set and $z_0\in\bbC\bs\fre$ a point which is fixed.  For any continuous, complex-valued function, $f$, on $\fre$ let
\begin{equation}\label{1.1}
  \norm{f}_\fre \equiv \sup_{z \in \fre} |f(z)|
\end{equation}
The residual polynomial, $R_{z_0,n}$, of $\fre$ normalized at $z_0$ is the unique polynomial that minimizes $\norm{P}_\fre$ over all polynomials, $P$, of degree at most $n$ with $P(z_0)=1$.  Such polynomials have been studied in numerical analysis as they have applications to the Krylov subspace iterations, see, for example, \cite{DTT98,Fis96,Kui06}.  Recently they have also been used to study the Remez inequality \cite{EicYud20}.  The residual norm is given by
\begin{equation}\label{1.2}
  r_{z_0,n} \equiv \norm{R_{z_0,n}}_\fre
\end{equation}
We will use $R_{z_0,n}^{(\fre)}$ and $r_{z_0,n}^{(\fre)}$ when we want to be explicit about the underlying set.  Since $R_{z_0,n}$ could be of degree less than $n$, $P=1$ and $R_{z_0,n-1}$ are trial polynomials and hence
\begin{align} \label{1.3}
  r_{z_0,n} \le r_{z_0,n-1} \le 1, \quad n\in\bbZ_+
\end{align}


These polynomials are clearly related to the Chebyshev polynomials, $T_{n}^{(\fre)}$, which minimize the sup norm over $\fre$, $t_{n}^{(\fre)}$, among all monic polynomials.  We will use heavily ideas from our papers \cite{CSZ1, CSYZ2, CSZ3, CSZ4} (the second joint with Yuditskii) discussing the asymptotics of such polynomials. While many of the extensions are direct, there are often subtle twists as we will see.

For the Chebyshev case, the dual problem of maximizing the leading coefficient of all degree $n$ polynomials with
\begin{equation}\label{1.2A}
  \norm{P_n}_\fre = 1
\end{equation}
is often useful and is trivially related to the minimization problem.  Similarly, here the dual problem of maximizing $P_n(z_0)$ over all degree at most $n$ polynomials with \eqref{1.2A} will play a role.  We will refer to the maximizers as dual residual polynomials and write them as  $\wti{R}_{z_0,n}$.

Basic to the theory is logarithmic potential theory (see \cite[Section 3.6]{HA} or \cite{ArmGar01, Hel09, Lan72, MF06, Ran95} for the basics of the subject). We will always assume that $\fre$ is a non-polar set and let $\rho_\fre$ and $g_\fre$ denote, respectively, the equilibrium measure and the Green's function of $\fre$.  In \cite{CSZ1,CSYZ2}, we  complexified the exponential Green's function $\exp[-g_\fre(z)]$, initially using a harmonic conjugate of $g_\fre$ near $z=\infty$, picking the branch so that this exponential looks like $Cz^{-1}+\text{O}(|z|^{-2})$ near infinity with $C>0$ and then analytically continuing.  This yields a multivalued analytic Blaschke-type function, $B_\fre$, on $(\bbC\cup\{\infty\})\bs\fre$ satisfying $|B_\fre(z)|=\exp[-g_\fre(z)]$.  In this paper, the $B_\fre$ we need will differ by a phase factor as we'll explain in detail in Section \ref{s5}. We normalize the the phase of $B_\fre$ so that $B_\fre(z_0)>0$.  

For a compact set $\fre\subset\bbC$, the outer domain, $\Omega$, of $\fre$ is the unbounded component of $(\hatC)\bs\fre$ and the outer boundary $O\pd(\fre)$ of $\fre$ is defined to be the boundary $\pd\Omega$.  The set $\hat \fre=(\hatC)\bs\Omega$ is called the polynomial convex hull of $\fre$.  Its boundary, $\pd\hat\fre$, coincides with the outer boundary $O\pd(\fre)$.  The Green's function, $g_\fre$, is positive on $\Omega$ and vanishes on the interior of $\hat\fre$, see for example \cite[Chapter~I.4]{SafTot97}.  If $\fre$ is regular for potential theory (which we usually assume), $g_\fre$ vanishes on all of $\hat\fre$ and hence $\hat\fre=\{z\in\bbC \,|\, g_\fre(z)=0\}$ and $\Omega=\{z\in\bbC\,|\,g_\fre(z)>0\}$.

For many years, the most striking aspect of the asymptotics of Chebyshev polynomials has been Widom's great 1969 discovery \cite{Wid69} that the suitably renormalized norms and asymptotics are almost periodic rather than a single limit.  Not surprisingly, our most important result here is that suitably renormalized $r_{z_0,n}$ and $R_{z_0, n}(z)$ are almost periodic, something which has not been hinted at in prior literature on residual polynomials (for the related Ahlfors problem, results of this type have been obtained by Eichinger--Yuditskii in \cite{EicYud18}).  In the work of Widom \cite{Wid69} on asymptotics of Chebyshev polynomials, a key object is $t_n/C(\fre)^n$, which, following Goncharov--Hatino\v{g}lu \cite{GH15}, have come to be called Widom factors.  In our situation the right analog, which we will still call Widom factors, are
\begin{equation}\label{1.4}
   W_n(\fre,z_0) \equiv r_{z_0,n}\,\left(e^{ng_\fre(z_0)}+e^{-ng_\fre(z_0)}\right), \quad n\in\bbN
\end{equation}
(see \eqref{1.5} below for why this is the correct normalization).

Section \ref{s2} will discuss a few general results on the general complex case including uniqueness of the minimizer and root asymptotics for $r_{z_0, n}$ and $|R_{z_0, n}|$.  There will also be a universal lower bound analogous to Szeg\H{o}'s result that $t_n\ge C(\fre)^n$.  Instead we will show that
\begin{equation}\label{1.4A}
  r_{z_0,n} \ge \exp[-ng_\fre(z_0)]
\end{equation}
a result that has appeared many times in the literature.

Most of the remainder of the paper focuses on the case when $\fre$ and $z_0=x_0$ are real.  In the Chebyshev case with $\fre\subset\bbR$, a critical role is played by the alternation theorem which goes back to Borel \cite{Borel} and Markov \cite{Markov}.  The version for residual polynomials, found by Achieser (aka Akhiezer) \cite{Ach} in 1932, is subtly different.  Section \ref{s3} begins with a proof of this result for the reader's convenience (given that the only proof we know in the literature is not readily available and not in English).  We then discuss a variety of applications.  Two unique to this situation (i.e., not relevant in the Chebyshev case) are the fact that $d_n \equiv \deg(R_{x_0, n})$ is always $n$ or $n-1$ (in the general complex case, the degree might be $0$) and that the dual polynomial, $\wti{R}_{x_0, n}$, is the same for all $x_0$ in the same connected component of $\bbR\bs\fre$, so, in particular, it equals the dual Chebyshev polynomial when $x_0$ is in either of the unbounded components.

As in \cite{CSZ1} for the Chebyshev case, the alternation theorem will let us show that
\begin{equation}\label{1.4B}
  \fre_n \equiv R_{x_0,n}^{-1}\left([-r_{x_0,n},r_{x_0,n}]\right)
\end{equation}
(where we emphasize that we mean the inverse as a map from $\bbC$ to itself) is a subset of $\bbR$.  This makes it what we called a period-$d_n$ set in \cite{CSZ1}, the spectrum of a period $d_n$ Jacobi matrix, which allows many detailed results.  In particular, we will prove that
\begin{equation}\label{1.5}
  2 \le W_n(\fre,x_0) \le 2 \exp[\PW(\fre,x_0)],\quad n\in\bbN
\end{equation}
where $\PW(\fre,x_0)$ is the Parreau--Widom constant of $\fre$ defined in Section~\ref{s3}.  The lower bound in \eqref{1.5} is due to Schiefermayr \cite{Sch11, Sch17}.  The upper bound is new here although it is an analog (with similar proof) of a result we proved for Chebyshev polynomials in \cite{CSZ1}.  Both inequalities are sharp and there are even cases where they are exact asymptotically for the $\liminf$ and $\limsup$!

Section \ref{s4} will discuss various interesting examples and includes a discussion of when $\deg(R_{x_0, n})$ is $n-1$.

Finally, Section \ref{s5} proves Szeg\H{o}--Widom asymptotics. To explain the main result of that section, we briefly recall what we called the Widom surmise in \cite{CSZ1}.  The two classical cases of Szeg\H{o} asymptotics \cite{L2Sz} concern limits of $z^{-n}P_n(z)$,  $z\notin\bbD$ for OPUC and of $[(z+\sqrt{z^2-4})/2]^{-n}P_n(z)$, $z\notin [-2,2]$ for OPRL whose measures have $[-2,2]$ as essential support.  The limit is the Szeg\H{o} function which is the solution of a minimization problem.  The prefactor in both cases is exactly what we called $B_\fre(z)^n$ and in both examples $\fre$ has capacity $1$ so a careful analysis suggests that one include a factor of $C(\fre)^n$. Indeed, Faber \cite{Fab1919} proved that for $\fre$ a connected and simply connected set (with analytic boundary), the Chebyshev polynomials, $T_{n}^{(\fre)}$, have what has come to be called Szeg\H{o} asymptotics (even though Faber's paper was earlier than Szeg\H{o}'s paper on OPUC asymptotics!), namely, that
\begin{equation}\label{1.6}
  B_\fre(z)^n T_n(z)/C(\fre)^n \to 1
\end{equation}
uniformly for $z$ in compact subsets of $\Omega$.  Widom realized that \eqref{1.6} cannot hold when $\Omega$ is not simply connected because the left side is not analytic on $\Omega$; rather, it is multivalued analytic.  Indeed, there is a character, $\chi_\fre$, of the fundamental group of $\Omega$  so that $B_\fre(z)$ is character automorphic (we'll recall what that means in Section \ref{s5}) with that character.  We will call by the name Widom minimizer the character automorphic function, $F(z;x_0,\chi)$, which is the unique (by arguments in Widom \cite{Wid69}, see also \cite{CSYZ2}) function with $F(z=x_0)=1$ minimizing the sup norm over $\Omega$, $\norm{F}_\Omega$.  Fixing $x_0$, we'll use $F_n$ for $F(\dott,x_0,\chi_\fre^n)$.  The Widom surmise says that in the Chebyshev case, the difference of the left side of \eqref{1.6} and $F_n$ goes to zero uniformly on $\Omega$.   Our analog replaces $C(\fre)$ by $e^{-g_\fre(x_0)}$.  As in the Chebyshev case, we'll require that $\fre$ obeys two conditions: Parreau--Widom (PW set) with a Direct Cauchy Theorem (DCT), notions we will also recall in later sections.  It says:

\begin{theorem} \lb{T1.1} Let $\fre\subset\bbR$ be a compact PW set with DCT and let $x_0\in\bbR\bs\fre$. Then
\begin{equation} \label{1.7}
  \lim_{n\to\infty} \left[ e^{n g_\fre(x_0)}\norm{R_{x_0,n}}_\fre - 2\norm{F_n}_\Omega \right] = 0
\end{equation}
and uniformly for $z$ in compact subsets of $\Omega$,
\begin{equation} \label{1.8}
  \lim_{n\to\infty} \left[ e^{n g_\fre(x_0)} B_\fre(z)^n R_{x_0,n}(z) - F_n(z) \right] = 0
\end{equation}
\end{theorem}

\begin{remarks} 1. Since the functions are multivalued, the proper formulation should talk about analytic functions on the universal cover of $\Omega$.  We could just as well put in branch cuts and discuss the functions and their boundary values on the cuts.  Since the difference is character automorphic, convergence in this cut region implies it on the universal cover.

2.  It is easy to see that \eqref{1.7} is equivalent to
\begin{equation}
\lim_{n\to\infty} \left[ W_n(\fre,x_0) - 2\norm{F_n}_\Omega \right] = 0
\end{equation}

3. Recall that $B_\fre$ is normalized by $B_\fre(x_0)>0$.  This implies that the quantity on the left vanishes identically at $x_0$.

4. Since $\norm{B_\fre}_\Omega = 1$, \eqref{1.8} might suggest that \eqref{1.7} holds without the factor of $2$.  That they aren't incompatible is because \eqref{1.8} doesn't hold uniformly as one approaches $\fre = \pd\Omega$.  One can partly understand where the factor of $2$ comes from by looking at the extra term we took in defining the Widom factor, \eqref{1.4}, which suggests that one should instead write \eqref{1.8} as
\begin{equation}\label{1.9}
  \lim_{n\to\infty} \left[ e^{n g_\fre(x_0)} R_{x_0,n}(z)B_\fre(z)^n - F_n(z)(1+B_\fre(z)^{2n}) \right] = 0
\end{equation}
Away from $\fre$, $B_\fre(z)^{2n}$ is negligible, it is not  on $\fre$, and, if the phase is coherent, the limit can be twice as large.

5. It is Widom \cite{Wid69} who noticed that for Chebyshev polynomials on $\fre=[-1,1]$ (which, up to a normalization constant, are the classical Chebyshev polynomials of the first kind), the analog of Theorem~\ref{T1.1} holds.  He conjectured it holds in general for finite gap sets but was only able to prove \eqref{1.7}.  That \eqref{1.8} holds for general finite gap sets was the main result of \cite{CSZ1}.  The proof there used Widom's result rather than rederiving it.  Fortunately, in \cite{CSYZ2}, we found a proof of both facts (for more general PW/DCT sets). Because we don't a priori have the analog of Widom's result for residual rather than Chebyshev polynomials, it is the approach of \cite{CSYZ2} that we'll adapt to the setting of residual polynomials in this paper.
\end{remarks}

\section{Basic Results} \lb{s2}

We begin with the case of $\fre$ in general position in $\bbC$.  As we noted, there are dual residual polynomials, $\wti{R}_{z_0, n}$, and special values, $\wti{r}_{z_0, n} = \wti{R}_{z_0, n}(z_0)$.  It is easy to see that the direct and dual problems are related by
\begin{equation}\label{2.1}
  \wti{R}_{z_0, n} = R_{z_0, n}/r_{z_0, n} ; \quad\; \wti{r}_{z_0, n} = 1/r_{z_0, n}
\end{equation}

There are two main results for the general case that we want to mention, uniqueness and root asymptotics.  An \emph{extreme point} for a polynomial $P$ is a point $z\in\fre$ for which $|P(z)|=\norm{P}_\fre$.

\begin{theorem} \lb{T2.1} \begin{SL}
\item[\rm{(a)}] Any residual polynomial, $R_{z_0,n}^{(\fre)}$, has at least $n+1$ extreme points.

\item[\rm{(b)}] The degree $n$ residual problem has a unique solution.
\end{SL}
\end{theorem}

\begin{remarks} 1. There can be infinitely many extreme points, see Example~\ref{E2.3}.

2. This extends the argument given in \cite{CSZ4} for Chebyshev polynomials and is well known.
\end{remarks}

\begin{proof} (a)  We claim that any norm minimizer, $P$, of degree at most $n$ with $P(z_0)=1$ must have at least $n+1$ extreme points.  For, if there are only $z_1,\dots,z_k$ with $k\le n$ distinct extreme points for $P$, then, by Lagrange interpolation, we can find a polynomial $Q$ of degree $k$ so that $Q(z_0)=0$ and $Q(z_j)=P(z_j)$, $j=1,\dots,k$.
Then for $\varepsilon$ small and positive, it is easy to see that $(P-\varepsilon Q)(z_0)=1$ and $\norm{P-\varepsilon Q}_\fre < \norm{P}_\fre$ violating the fact that $P$ is a norm minimizer since $\deg(P-\varepsilon Q)\le\deg(P)\le n$.

(b) Suppose now that $P$ and $Q$ are both norm minimizers among polynomials of degree at most $n$ taking the value $1$ at $z_0$.  Then so is $R=\frac{1}{2}(P+Q)$.  Pick $\{z_j\}_{j=1}^{n+1}$ distinct extreme points for $R$.  Since $|R(z_j)|=r_{z_0,n}$ and $|P(z_j)|, |Q(z_j)|\le r_{z_0,n}$, we must have that $P(z_j)=Q(z_j)$ for $j=0,1,\dots,n+1$.  As $\deg(P-Q)\le n$, we have that $P=Q$ completing the proof of uniqueness of the minimizing polynomial.
\end{proof}

The first assertion in the following, which we'll need for root asymptotics, appears many times in the literature.  The second assertion is an analog of a result we proved for Chebyshev polynomials in \cite{CSZ3}.

\begin{theorem} \lb{T2.2}
Let $\fre\subset\bbC$ be a compact non-polar set and $z_0\in\bbC\bs\fre$.  Then for all $n\in\bbN$,
\begin{equation} \label{2.2}
  \norm{R_{z_0,n}}_\fre \ge \exp[-ng_\fre(z_0)]
\end{equation}
The equality in \eqref{2.2} is attained for some $n=n_0\in\bbN$ if and only if there exists a polynomial $P$ of degree $n_0$ such that
\begin{equation} \label{2.3}
  O\pd(\fre) = P^{-1}(\pd\bbD)
\end{equation}
In this case, $R_{z_0,n_0}$ is of degree $n_0$ and equality in \eqref{2.2} is attained for $n=kn_0$ for all $k\in\bbN$.
\end{theorem}
\begin{proof}
The lower bound \eqref{2.2} follows from the Bernstein--Walsh inequality (\cite[Chap.~III, Eq.~(2.4)]{SafTot97} or \cite[Theorem~3.7.1]{HA}) for $R_{z_0,n}$,
\begin{equation} \label{2.4}
  \frac{|R_{z_0,n}(z)|}{\norm{R_{z_0,n}}_\fre} \le \exp[ng_\fre(z)],\quad z\in\bbC
\end{equation}
evaluated at $z=z_0$.

Suppose equality is attained in \eqref{2.2} for $n=n_0$.  Consider the dual polynomial $P(z):=R_{z_0,n_0}(z)/\norm{R_{z_0,n_0}}_\fre$.  Then $\norm{P}_{\hat\fre}=1$ and hence
\begin{equation} \label{2.5}
  \{z:|P(z)|>1\}\subset\Omega
\end{equation}
By \eqref{2.4}, the function $PB_\fre^n$, which is multivalued analytic on $\Omega$, has $\norm{PB_\fre^n}_{\Omega}\le1$.  Since $B_\fre(z_0)>0$, the assumption of equality in \eqref{2.2} implies $P(z_0)B_\fre(z_0)^n=1$ and hence, by the maximum principle,
\begin{equation}\label{2.5A}
   PB_\fre^n=1 \text{ on } \ol\Omega
\end{equation}
Thus on $\ol\Omega$, we have that $|P(z)|=\exp\left[ng_\fre(z)\right]$ and hence
\begin{equation} \label{2.6}
  \Omega\subset\{z:g_\fre(z)>0\}\subset\{z:|P(z)|>1\}
\end{equation}
It follows from \eqref{2.5} and \eqref{2.6} that $\{z:|P(z)|>1\}=\Omega$ and hence also the boundaries of the sets are equal which is \eqref{2.3}.  Moreover, by \eqref{2.5A} and the leading $Cz^{-1}$ behavior of $B_\fre(z)$ near infinity, one sees that $\deg(P)=\deg(R_{z_0,n_0})=n_0$.

Conversely, suppose \eqref{2.3} holds.  Then $g_\fre(z)=\tfrac{1}{n_0}\log|P(z)|$ and so $\exp[-kn_0g_\fre(z_0)]=|P(z_0)|^{-k}$.  Let $Q(z)=[P(z)/P(z_0)]^k$.  Then $Q$ is of degree $kn_0$, $Q(z_0)=1$, and $\norm{Q}_\fre=|P(z_0)|^{-k}=\exp[-ng_\fre(z_0)]$ since $\norm{P}_\fre=1$.  Using $Q$ as a trial polynomial, one has that $\norm{R_{z_0,n}}_\fre\le\norm{Q}_\fre$,  which implies equality in \eqref{2.2} for $n=kn_0$.
\end{proof}

\begin{example} \lb{E2.3} Let $\fre=\ol{\bbD}$ and $|z_0|>1$.  Then by the above, $R_{z_0,n}(z)=(z/z_0)^n$ so the dual residual polynomial is equal to the dual Chebyshev polynomial (which is also the Chebyshev polynomial).  Moreover, every point in $\partial\fre$ is an extreme point showing there are infinitely many such points.  Indeed, the above shows this is true for any lemniscate (i.e., set of the form \eqref{2.3}).
\end{example}

\begin{theorem} \lb{T2.4} Let $\fre\subset\bbC$ be a compact non-polar set and $z_0\in\bbC\bs\fre$.  Then
\begin{SL}
\item[\rm{(a)}]
\begin{equation} \label{2.7}
  \lim_{n\to\infty}\norm{R_{z_0,n}}_\fre^{1/n} = \exp[-g_\fre(z_0)]
\end{equation}

\item[\rm{(b)}] If $K$ is a closed set containing all zeros of $R_{z_0,n}$ for large $n$, but not the point $z_0$, and so that $\bbC \bs K$ is connected, then
\begin{equation}\label{2.8}
  | R_{z_0, n}(z) |^{1/n} \to \exp [g_\fre(z)-g_\fre(z_0)]
\end{equation}
uniformly on compact subsets of $\bbC \bs K$.

In particular, if $\fre\subset\bbR$, $z_0\in\bbR\bs \fre$ and $(\al,\be)\subset\bbR\bs\fre$ is an interval not containing any zeros of $R_{z_0, n}$ for large $n$ (this always holds for the gap of $\fre$ containing $z_0$), then \eqref{2.8} holds uniformly on compact subsets of $(\bbC \bs \bbR)\cup(\al,\be)$.
\end{SL}
\end{theorem}

\begin{remarks} 1. The analog of (a) for Chebyshev polynomials is sometimes called the Faber--Fekete--Szeg\H{o} theorem after \cite{Fab1919, Fek1923, Sze1924}; the result for residual polynomials appears many times in the literature, see for example \cite{DTT98, Kui06}.  (b) would be expected by any attentive reader of Stahl--Totik \cite{StaTot92} or Saff--Totik \cite{SafTot97}; our proof here is essentially the same as the analog for Chebyshev polynomials in \cite{CSZ1}.

2. The ``in particular'' assertion in (b) uses facts proven in the next section and is immediate given those facts.
\end{remarks}

\begin{proof} (a) Let $Q_n$ with $\deg(Q_n)=n$ be Fekete polynomials \cite[Definition~5.5.3]{Ran95} for the set $\fre$.  Then, by \cite[Theorems~5.5.2 and 5.5.7]{Ran95}, locally uniformly on $\bbC\bs\fre$, one has that
\begin{equation} \lb{2.9}
  \left(\frac{|Q_n(z)|}{\norm{Q_n}_\fre}\right)^{1/n} \to \exp[g_\fre(z)]
\end{equation}
Since $R_{z_0,n}$ is a norm minimizer, $\norm{R_{z_0,n}}_\fre \le \norm{Q_n/Q_n(z_0)}_\fre$ and hence
\begin{equation} \lb{2.10}
  \limsup_{n\to\infty}\, \norm{R_{z_0,n}}_\fre^{1/n} \le
  \limsup_{n\to\infty} \left(\frac{\norm{Q_n}_\fre}{|Q_n(z_0)|}\right)^{1/n} =
  \exp[-g_\fre(z_0)]
\end{equation}
Combined with the lower bound \eqref{2.2}, this yields \eqref{2.7}.

(b) It follows from the Bernstein--Walsh inequality \eqref{2.4} that
\begin{align}
  | R_{z_0, n}(z) |^{1/n} \leq || R_{z_0, n} ||_\fre^{1/n} \exp [g_\fre(z)]
\end{align}
Moreover, by assumption, for $n$ sufficiently large all the zeros of $R_{z_0, n}$ lie is $K$. Thus, the function
\begin{align}
  h_n(z) = \frac{1}{n} \log || R_{z_0, n} ||_\fre + g_\fre(z) - \frac{1}{n} \log | R_{z_0, n}(z) |
\end{align}
is non-negative and harmonic on $\bbC\bs K$.  Since $h_n(z_0)\to 0$ by part (a), Harnack's inequality implies that $h_n\to 0$ uniformly on compact subsets of $\bbC\bs K$.
\end{proof}

\section{Alternation Theorem and Consequences} \lb{s3}

Recall the following definition and theorem of Borel \cite{Borel} and Markov \cite{Markov} critical for the understanding of Chebyshev polynomials on subsets of $\bbR$.

\begin{definition} Let $\fre \subset \bbR$ be compact.  We say that $P_n$, a degree $n$ polynomial, has an \textit{alternating set} in $\fre$ if there exists $n+1$ points, $\{x_j\}^{n+1}_{j=1} \subset \fre$, with $x_1 < x_2 < \ldots < x_{n+1}$ so that
\begin{equation} \lb{3.1}
   P_n(x_j) = (-1)^{n+1-j} \norm{P_n}_\fre,\quad j=1,\ldots,n+1
\end{equation}
\end{definition}

\begin{theorem} [The Alternation Theorem for Chebyshev Polynomials] \lb{T3.1}
Let $\fre \subset \bbR$ be compact.  The Chebyshev polynomial of degree $n$ for $\fre$ has an alternating set in $\fre$.  Conversely, any monic polynomial with an alternating set in $\fre$ is the Chebyshev polynomial for $\fre$.
\end{theorem}

For a proof, see \cite{CSZ1}. The analog for residual polynomials is due to Achieser \cite{Ach}.

\begin{definition} Let $\fre \subset \bbR$ be compact and $x_0\in\bbR\bs\fre$.  We say that $P_n$, a degree at most $n$ polynomial, has an \textit{$x_0$-alternating set} in $\fre$ if there exists $n+1$ points, $\{x_j\}^{n+1}_{j=1} \subset \fre$, with $x_1 < x_2 < \ldots < x_k<x_0<x_{k+1}<\ldots<x_{n+1}$ for some $k\in\{0,1,\ldots, n+1\}$ so that
\begin{equation} \lb{3.2}
   P_n(x_j) = (-1)^{k+1-j} \sgn(x_j-x_0) \norm{P_n}_\fre, \quad j=1,\ldots,n+1
\end{equation}
\end{definition}

\begin{remark} $k=0$ (resp. $k=n+1$) means that $x_0<x_1$ (resp. $x_{n+1}<x_0$) and, in that case (up to a possible sign change), \eqref{3.2} is the same as \eqref{3.1}.
\end{remark}

\begin{theorem} [The Alternation Theorem for Residual Polynomials~\cite{Ach}] \lb{T3.2}
Let $\fre \subset \bbR$ be compact and $x_0\in\bbR\bs\fre$.  The residual polynomial, $R_{x_0,n}$, of degree at most $n$ for $\fre$ has an $x_0$-alternating set in $\fre$.  Conversely, any polynomial, $P$, with an $x_0$-alternating set in $\fre$ and with $P(x_0)=1$ is the $R_{x_0,n}$ polynomial for $\fre$.
\end{theorem}

\begin{proof} Suppose that $P$ is a polynomial with $P(x_0)=1$ and so that \eqref{3.2} holds.  If $P$ is not a norm minimizer, then $\norm{R_{x_0,n}}_\fre<\norm{P}_\fre$.  Consider the polynomial $Q(x)=[P(x)-R_{x_0,n}(x)]/(x-x_0)$.  It has degree at most $n-1$ and alternating signs at $x_1,\dots,x_{n+1}$, hence a zero in each of the intervals $(x_j,x_{j+1})$, $j=1,\dots,n$.  It follows that $Q$ is identically zero, a contradiction.  Thus $P$ is a norm minimizer.

Conversely, suppose $P\equiv R_{x_0,n}$ and $\sgn(x-x_0)P(x)$ has at most $n-1$ sign changes on the set of extreme points of $P$.  Then, by putting zeros in the right places, there exists a polynomial $Q_0$ of degree at most $n-1$ such that $\sgn(Q_0(x))=\sgn(P(x)/(x-x_0))$ for each extreme points of $P$.  Thus the polynomial $Q(x)=(x-x_0)Q_0(x)$ has degree at most $n$, $(P-\varepsilon Q)(x_0)=1$, and $\norm{P-\varepsilon Q}_\fre<\norm{P}_\fre$ for sufficiently small $\varepsilon>0$ contradicting the fact that $P$ is a norm minimizer.  Therefore an alternating set exists.
\end{proof}

This shows that $R_{x_0,n}$ has at least $n+1$ extreme points so proving uniqueness again.  Uniqueness (and consideration of $\ol{R_{x_0,n}(\bar{x})}$) shows that all the coefficients of $R_{x_0,n}$ are real.  Besides this, the Alternation Theorem has lots of immediate corollaries, many of which are so important that we will call them theorems. For the rest of this section, we will suppose that $\fre \subset \bbR$ is compact and $x_0\in\bbR\bs\fre$.  We will let $x_\pm$ be the $\sup$/$\inf$ of $\fre$ and $k$ the integer specified in the definition of $x_0$-alternating set.

\begin{theorem} \lb{T3.3} The dual residual polynomials $\wti{R}_{x_0,n}$ do not change as $x_0$ varies through the same connected component of $\bbR\bs\fre$ and are equal (up to a constant) to the Chebyshev  polynomials (of the same order) if $x_0\in (-\infty,x_-)\cup (x_+,\infty)$.  Equivalently, if $x_0, y_0$ are in the same connected component of $\bbR\bs\fre$, we have that
\begin{equation}\label{3.3}
  R_{y_0,n}(x) = R_{x_0,n}(x)/R_{x_0,n}(y_0)
\end{equation}
\end{theorem}

\begin{proof} As we noted already, if $x_0\in\bbR\bs [x_-,x_+]$, then an $x_0$-alternating set is an alternating set so, by Theorem~\ref{T3.1}, up to a constant, $R_{x_0,n}$ is the ordinary Chebyshev polynomial.  It follows that $R_{x_0,n}$ is non-vanishing outside $[x_-,x_+]$ so \eqref{3.3} holds.  In (a) of Theorem~\ref{T3.4}, we'll prove (without using this theorem) that if $x_0$ is in a bounded component of $\bbR\bs\fre$, then $R_{x_0,n}$ is non-vanishing on that component so for $y_0$ in the same component, the polynomial on the right side of \eqref{3.3} has a $y_0$-alternating set and is normalized properly and so is $R_{y_0,n}$.  It is easy to see that \eqref{3.3} proves equality of the dual residual polynomials.
\end{proof}

As we noted, if $k=0$ or $k=n+1$ then $R_{x_0,n}$ is a constant multiple of the ordinary Chebyshev polynomial whose structure we know. In particular, for the ordinary Chebyshev polynomial the $\inf$ and $\sup$ of $\fre$ are the alternation points $x_1$ and $x_{n+1}$, respectively, and hence $k=0$ (resp. $k=n+1$) happens if and only if $x_0\in (-\infty,x_-)$ (resp. $x_0\in(x_+,\infty))$. Henceforth, we will suppose that $1\le k\le n$ and that $x_0$ lies in a bounded component, $(\alpha,\beta)$, of $\bbR\bs\fre$.

\begin{theorem} \lb{T3.4}
\begin{SL}
\item[\rm{(a)}] $R_{x_0,n}(x)$ has at least $n-1$ zeros in $(x_-,x_k)\cup(x_{k+1},x_+)$ and no zeros in $[x_k,x_{k+1}]$.

\item[\rm{(b)}] $d_n\equiv \deg(R_{x_0,n})$ is either $n$ or $n-1$.

\item[\rm{(c)}] All zeros of $R_{x_0,n}(x)$ are real and simple.
\end{SL}
\end{theorem}

\begin{proof} (a) There are $k-1$ disjoint intervals, $(x_1,x_2),\ldots,(x_{k-1},x_k)$
in $(x_-,x_k)$, each with an odd number of zeros (counting multiplicity) and similarly, $n-k$ such intervals in $(x_{k+1},x_+)$.  Furthermore, there is an even number of zeros (counting multiplicity and including $0$ as even) in $(x_k,x_{k+1})$.  Since there are $n-1$ odd zero intervals and at most $n$ zeros, each odd number interval must have exactly one zero, so simple, and $(x_k,x_{k+1})$ cannot have any zeros.

(b) We have proven that $R_{x_0,n}$ has at least $n-1$ zeros so its degree must be at least $n-1$ and is, by definition, at most $n$.

(c) We have proven in (a) that $R_{x_0,n}$ has at least $n-1$ simple real zeros.  If it is of degree $n$, its last zero must also be simple and, since $R_{x_0, n}$ is real and non-real zeros come in complex conjugate pairs, the $n$-th zero must also be real.
\end{proof}

\begin{theorem} \lb{T3.5}
Suppose that $1<k<n$.
\begin{SL}
\item[\rm{(a)}] The derivative, $R_{x_0,n}'(x)$, has at least $k-2$ zeros (counting multiplicity) in $(x_-,x_k)$, at least $n-k-1$ zeros in $(x_{k+1},x_+)$, and at least one zero in $[x_k,x_{k+1}]$.

\item[\rm{(b)}] All zeros of $R_{x_0,n}'(x)$ are real and simple.

\item[\rm{(c)}] $R_{x_0,n}(x) > r_{x_0,n}$ on $(x_k,x_{k+1})$.

\item[\rm{(d)}] $x_k=\alpha$ and $x_{k+1}=\beta$.

\item[\rm{(e)}] Either $x_1=x_-$ or $x_{n+1}=x_+$ (or both).
\end{SL}
\end{theorem}

\begin{proof} (a) We start with the assertion that $R_{x_0,n}'$ has at least $n-k-1$ zeros in $(x_{k+1},x_+)$.  There is nothing to prove unless $n\ge k+2$.  Since $R_{x_0,n}(x_{k+1})=R_{x_0,n}(x_{k+3})=r_{x_0,n}$ and $R_{x_0,n}(x_{k+2})=-r_{x_0,n}$, $R_{x_0,n}$ must have at least one minimum point in $(x_{k+1},x_{k+3})$ and so a point where $R_{x_0,n}'$ vanishes.  If $n=k+2$, we have the one required zero.  If $n>k+2$, we can find a maximum in $(x_{k+2},x_{k+4})$.  Proceeding in this way we get the required $n-k-1$ zeros of $R_{x_0,n}'$ and similarly, we get at least $k-2$ zeros in $(x_-,x_k)$.  This accounts for at least $n-3$ zeros out of a maximum possible $n-1$ (which is the maximum degree of $R_{x_0,n}'$).

We have that $R_{x_0,n}(x_k)=R_{x_0,n}(x_{k+1})=r_{x_0,n} < 1=R_{x_0,n}(x_0)$ so $R_{x_0,n}$ has a maximum point in $(x_k,x_{k+1})$ and so a point where $R_{x_0,n}'$ must vanish.

(b) All the zeros of $R_{x_0,n}'$ found in (a) are local maxima or local minima and so points where $R_{x_0,n}'$ has odd order zeros. Since there are $n-2$ such points, none can have a zero of order $3$ or more and thus all must be simple.  The remaining zero must also be simple and, by reality of $R_{x_0,n}'$, real.

(c) The same argument (from the proof of Theorem~\ref{T3.5}~(a)) that proved $R_{x_0,n}$ has no zeros in $(x_k,x_{k+1})$, shows it cannot take any value in $(-r_{x_0,n},r_{x_0,n})$.  Since that interval is connected and $R_{x_0,n}(x_0)=1>r_{x_0,n}$, we conclude that on that interval we have that $R_{x_0,n}(x) \ge r_{x_0,n}$.  If there were a point, $y_0$, in the open interval with $R_{x_0,n}(y_0)=r_{x_0,n}$, then on the interval, $R_{x_0,n}$ would have at least one local minimum (at $y_0$) and two local maxima (by Rolle's theorem), so three zeros of $R_{x_0,n}'$ in the interval which cannot happen because we've found $n-3$ zeros outside that interval out of at most $n-1$ in total.

(d) Since $\alpha\in [x_k,x_0)$ and $R_{x_0,n}(\alpha)\le r_{x_0,n}$ (because $\alpha\in\fre))$, (c) implies that $\alpha=x_k$.  Similarly, $\beta=x_{k+1}$.

(e) We will defer the proof of this to later. See the third remark after Proposition \ref{P3.7}.
\end{proof}

\begin{theorem} \lb{T3.6} If $\deg(R_{x_0,n})=n-1$, then
\begin{SL}
\item[\rm{(a)}] We have that $R_{x_0,n}(x)= T_{n-1}(x)/T_{n-1}(x_0)$, where $T_{n-1}$ is the Chebyshev polynomial for $\fre$ of degree $n-1$.

\item[\rm{(b)}] $R_{x_0,n-1}=R_{x_0,n}$.
\end{SL}
\end{theorem}

\begin{proof} (a) If we drop $x_{k+1}$, we get an ordinary alternating polynomial, so $R_{x_0,n}$ is a multiple of the Chebyshev polynomial $T_{n-1}$.

(b) If we drop $x_{n+1}$, we get a trial polynomial for the $R_{x_0,n-1}$ problem with an $x_0$-alternating set.
\end{proof}

Next, we turn to the idea of using $R_{x_0,n}$ to approximate $\fre$ with the spectra of periodic Jacobi matrices, an idea that was so useful in the Chebyshev case \cite{CSZ1}. Given $\fre\subset\bbR$ and $x_0\in\bbR\bs\fre$, we define the \emph{period-$d_n$ sets}
\begin{equation}\label{3.4}
  \fre_n^\circ=R_{x_0,n}^{-1}\bigl((-r_{x_0,n},r_{x_0,n})\bigr); \quad\; \fre_n=R_{x_0,n}^{-1}\bigl([-r_{x_0,n},r_{x_0,n}]\bigr)
\end{equation}
where we consider the inverses as maps from $\bbC$ to itself (so that, by the open mapping theorem, $\fre_n$ is the closure of $\fre_n^\circ$) although it will turn out the sets are subsets of $\bbR$.  We will denote the equilibrium measure, the Green's function, and the corresponding Blaschke-type function (normalized so that $B_n(x_0)>0$) of $\fre_n$ by $\rho_n$, $g_n$, and $B_n$, respectively.

\begin{proposition} \label{P3.7}
Let $\fre\subset\bbR$ and $x_0\in\bbR\bs\fre$.  Then
\begin{SL}
  \item[\rm{(a)}] The set $\fre_n^\circ$ is a subset of $\bbR$ and has $d_n$ connected components, each an interval of $\fre_n$-harmonic measure $1/d_n$.
  \item[\rm{(b)}] $\fre\subset\fre_n\subset\bbR$ and each gap of $\fre$, including the ``unbounded gap'' $(-\infty,x_-)\cup(x_+,\infty)$, intersects with at most one component of $\fre_n^\circ$.  The gap of $\fre$ containing $x_0$ does not intersect $\fre_n$.  If $d_n=n-1$ then also the unbounded gap does not intersect $\fre_n$.
  \item[\rm{(c)}] The equilibrium measure, $\rho_n$, of each gap of $\fre$ is at most $1/d_n$.
  \item[\rm{(d)}] In any component of $\bbR\bs\fre$, $R_{x_0,n}$ has at most one zero.
\end{SL}
\end{proposition}

\begin{remarks} 1. We recall that the $\frg$-\emph{harmonic measure} of a set $\frf\subset\frg$ is $\rho_\frg(\frf)$.

2. Since $\fre_n$ is bounded and components of $\fre_n^\circ$ are connected, by (b), either $\rho_n((-\infty,x_-))=0$ or $\rho_n((x_+,\infty))=0$ (or both).

3. Theorem \ref{T3.5} (e) follows from the proof of Proposition \ref{P3.7} (b).

4. It follows from (a) that $\fre_n$ consists of at most $d_n$ real intervals.  This fact was already proven in \cite{Sch11}.
\end{remarks}

\begin{proof}
As noted before, if $x_0\in\bbR\bs[x_-,x_+]$ then $R_{x_0,n}$ is a constant multiple of the Chebyshev polynomial, $T_{n}$. Likewise, if $d_n=n-1$ then, by Theorem~\ref{T3.6}~(a), the residual polynomial $R_{x_0,n}$ is a constant multiple of the Chebyshev polynomial, $T_{n-1}$. In those cases the Proposition follows from the corresponding result for Chebyshev polynomials  \cite[Theorems 2.3 and 2.4]{CSZ1}.  Thus, in the following  we assume $d_n=n$ and that $x_0$ lies in a bounded component of $\bbR\bs\fre$.

(a) Let $\{x_j\}_{j=1}^{n+1}\subset\fre$ be as in the alternation theorem, and let $1\le k\le n$ be so that $x_0\in(x_k,x_{k+1})$.  Then by \eqref{3.2} and the intermediate value theorem, the residual polynomial $R_{x_0,n}$ attains each value $y\in(-r_{x_0,n},r_{x_0,n})$ in the $n-1$ intervals $(x_j,x_{j+1})$, $j=1,\dots,k-1,k+1,\dots,n$ an odd number of times which accounts for all the pre-images of $y$ but one.  By reality of $R_{x_0,n}$, the remaining solution is also in $\bbR$ and so all values occur once proving that there are $n$ distinct intervals in $\fre_n^\circ$.  Hence $\fre_n^\circ$ lies on the real line and so, by the open mapping theorem, does $\fre_n$. By Theorem~\ref{T3.8} below, $\fre_n^\circ$ consists of $n$ components each of harmonic measure $1/n$.

(b) The analysis of taking the values in $(-r_{x_0,n},r_{x_0,n})$ shows that $\fre_n^\circ$ has alternating intervals where $R_{x_0,n}$ increases and decreases.  It follows that $R_{x_0,n}$, which has $2n$ points (counting multiplicity) where it takes the values $\pm r_{x_0,n}$, must have those points (listed in increasing order) as one of one sign, then successive pairs (counting multiplicity of solutions of $R_{x_0,n}\pm r_{x_0,n}=0$) of the same sign (opposite to the previous sign), ending in a singlet.  It is easy to see that to have the requisite number of sign changes, the alternating set must contain both extreme points at the end of the gap containing $x_0$, one from each of the remaining $n-2$ pairs and one of the two singlets.  It follows that each gap of $\fre_n^\circ$ (including the case of touching gaps) must have at least one end in $\fre$.  If some gap of $\fre$ contained parts of two bands from $\fre_n$, it would contain the closure of the entire gap between them which violates the conclusion we reached that such gaps have an endpoint in $\fre$.  Thus at most one component of $\fre_n$ intersects any gap of $\fre$ as claimed.  Theorem~\ref{T3.5}~(c) implies that the gap containing $x_0$ is disjoint from $\fre_n$.

(c) Since, by (b), each gap of $\fre$ intersects at most one component of $\fre_n^\circ$ and, by (a), the equilibrium measure $\rho_n$ of each component of $\fre_n^\circ$ is $1/n$, it follows that $\rho_n$ of each gap of $\fre$ is at most $1/n$.

(d) Between any two zeros of $R_{x_0,n}$, there are a pair of extreme points (counting multiplicity of solutions of $R_{x_0,n}\pm r_{x_0,n}=0$) with the same sign.  Since one of those points must lie in $\fre$, the two zeros can't lie in the same gap.
\end{proof}

What makes these sets useful is that there are explicit formulae for $g_n$ and $B_n$ in terms of $R_{x_0,n}$.  It will be convenient to introduce
\begin{equation} \label{3.5}
  \Delta_n(z) \equiv 2R_{x_0,n}(z)/\norm{R_{x_0,n}}_\fre
\end{equation}
Then

\begin{theorem} \lb{T3.8}
If $\fre\subset\bbR$ and $x_0\in\bbR\bs\fre$, then for all $z\in\bbC$ we have that
\begin{equation} \lb{3.6}
  g_n(z) = \frac{1}{d_n} \big(g_{[-2,2]}\circ\Delta_n\big)(z) = \frac{1}{d_n} \log\left|\frac{\Delta_n(z)}{2}+\sqrt{\left(\frac{\Delta_n(z)}{2}\right)^2-1}\,\right|
\end{equation}
This implies that each open component of $\fre_n^\circ$ has $\fre_n$-harmonic measure $1/n$.  Moreover
\begin{equation} \lb{3.7}
  B_n(z)^{\pm d_n} = \frac{\Delta_n(z)}{2}\mp\sqrt{\left(\frac{\Delta_n(z)}{2}\right)^2-1}
\end{equation}
and
\begin{equation} \label{3.8}
  \frac{2R_{x_0,n}(z)}{\norm{R_{x_0,n}}_\fre} = \Delta_n(z) = B_n(z)^{d_n}+B_n(z)^{-d_n}
\end{equation}
In particular, evaluating \eqref{3.8} at $z=x_0$ implies that
\begin{equation} \label{3.9}
  \norm{R_{x_0,n}}_\fre = 1/\cosh\bigl( d_n g_n(x_0) \bigr)
\end{equation}
\end{theorem}

\begin{remarks} 1. \eqref{3.9} shows why we normalize the Widom factor in \eqref{1.4} as we do; it shows that $W_n(\fre_n,x_0)=2$ (the Alternation Theorem implies that $R_{x_0,n}^{(\fre_n)}=R_{x_0,n}^{(\fre)})$.

2. This shows that $\fre_n$ is the spectrum of a periodic Jacobi matrix.  $\Delta_n$ is either its discriminant or its negative (depending on whether $R_{x_0,n}$ determined by $R_{x_0,n}(x_0)=1$ has leading positive or negative coefficient).

3. Recall that $B_n$ is normalized so that $B_n(x_0)>0$.
\end{remarks}

\begin{proof} Essentially, the same as \cite[Theorems 2.2 and 2.3]{CSZ1}.  The function in absolute value in \eqref{3.6} is analytic and non-vanishing on $\fre_n$ (where there is a square root branch cut) and has magnitude $1$ on $\fre_n$.  It is discontinuous across the cut but its magnitude is continuous. Thus the quantity on the right of that equation is harmonic on $\bbC\bs\fre_n$ with a logarithmic singularity at $\infty$.  Thus it is the Green's function as claimed.  The remainder is immediate given the calculation of the harmonic measure in \cite[Theorem~2.3]{CSZ1}.
\end{proof}

As a corollary, we get the lower bound of Schiefermayr \cite{Sch11, Sch17} by a proof which can be viewed essentially as a reworking of his proof  (the characterization of when equality holds is an extension of \cite{Sch11} as we do not apriori assume that $\fre$ is a finite union of intervals):

\begin{corollary} \lb{C3.9}
If $\fre\subset\bbR$ and $x_0\in\bbR\bs\fre$, then for all $n\in\bbN$,
\begin{equation} \label{3.10}
  \norm{R_{x_0,n}}_\fre \ge \frac{2}{e^{ng_\fre(x_0)}+e^{-ng_\fre(x_0)}}
\end{equation}
or equivalently, by \eqref{1.4},
\begin{equation} \lb{3.11}
  W_n(\fre,x_0) \ge 2
\end{equation}
Equality is attained in \eqref{3.10} if and only if $d_n=n$ and $\fre=\fre_n$ (equivalently, $\fre$ is a finite gap set with at most $n$ components, each of which has harmonic measure an integral multiple of  $1/n$).
\end{corollary}

\begin{proof}
Since $\fre\subset\fre_n$, we have that
\begin{equation}\label{3.12}
  g_\fre(z) \ge  g_n(z)
\end{equation}
for all $z\in\bbC$.
As $\cosh$ is monotone on $(0,\infty)$ and $d_n\le n$, it follows that $\cosh(ng_\fre(x_0)) \ge \cosh(d_n g_n(x_0))$.  Thus, by \eqref{3.9},
\begin{align}
  \norm{R_{x_0,n}}_\fre = 1/\cosh\bigl(d_n g_n(x_0)\bigr) \ge 1/\cosh\bigl(n g_{\fre}(x_0)\bigr)
\end{align}
which is \eqref{3.10}.

This also shows that equality holds if and only if $d_n=n$ and $g_n(x_0)=g_{\fre}(x_0)$ (because $\cosh$ is strictly monotone).  If $\fre_n=\fre$, the Green's functions are clearly equal at $x_0$.  Conversely, if one has equality at $x_0$, then the inequality \eqref{3.12} plus Harnack's theorem implies equality of the Green's functions in the upper half-plane. Since $\fre_n$ is a polynomial pre-image of an interval, $\fre_n$ is regular for potential theory and so $g_n$ is continuous on $\bbC$ and vanishes on $\fre_n$. Then $g_n=g_\fre$ in the upper half-plane implies that $g_n=g_\fre>0$ on $\bbC\bs\fre$ and hence $\fre_n \subset \fre$. Since $\fre \subset \fre_n$ by construction of $\fre_n$, equality of the sets follows.
\end{proof}

We now turn to the question of upper bounds on $\norm{R_{x_0,n}}_\fre$ that differ from the lower bounds by only a constant.  For norms of Chebyshev polynomials, such upper bounds were proven for $\fre\subset\bbR$, a compact finite gap set, by Widom \cite{Wid69} and Totik \cite{Tot09}.  While these bounds did not have explicit constants, in \cite{CSZ1} we called new bounds with explicit (and in many cases optimal) constant, \emph{Totik--Widom upper bounds} and we use that name here for analogous bounds of residual polynomials even though, so far as we know, there are no prior such bounds in the literature.  These require us to define suitable Parreau--Widom constants for a compact set $\fre\subset\bbC$ relative to a point $z_0\in\bbC\bs\fre$ with $g_\fre(z_0)>0$ (this last condition is true if and only if $z_0$ is in the unbounded component of $\bbC\bs\fre$).  Introducing the set $\frf:=\left\{(z-z_0)^{-1}:z\in\fre\right\}$, we have that $g_\fre(z,z_0) \equiv g_\frf((z-z_0)^{-1})$ is the Green's function for $\fre$ with singularity at $z_0$. We note the well-known fact (\cite[(3.8.46)]{HA}) that $g_\fre(z,w)=g_\fre(w,z)$ so, in particular,
\begin{equation}\label{3.12A}
    g_\fre(\infty,z_0)=g_\fre(z_0)
\end{equation}
Define
\begin{equation} \lb{3.13}
  \calC = \left\{z_0+z^{-1} : z\in\bbC\bs\frf \text{ such that } \nabla g_\frf(z)=0\right\}
\end{equation}
This means $\calC$ is precisely the set of ordinary finite critical points of $g_\fre(\dott,z_0)$ plus infinity if it is a critical point in local coordinates.  We define the PW constant for $\fre$ relative to $z_0$ by
\begin{equation} \lb{3.14}
  \PW(\fre,z_0) \equiv \PW(\frf) = \sum_{c_j\in\calC} g_\fre(c_j,z_0)
\end{equation}
It is known that $\PW(\fre,z_0)<\infty$ if and only if $\PW(\fre)<\infty$.  See, e.g., Hasumi \cite[Chapter~V]{Hasu}.  If that holds, we say that $\fre$ is a PW set.  As with $g_\fre(z)$, for any $w\in\bbC\bs\fre,\, w\ne x_0$, we define the Blaschke type function, $B_\fre(z,w)$, as the unique multivalued analytic function with $|B_\fre(z,w)|=\exp\left[-g_\fre(z,w)\right]$ and $B_\fre(x_0,w)>0$.

Turning to the real case, we recall that if $\frf\subset\bbR$ is compact, then all critical points of $g_\frf$ are in $\bbR$. Since $g_\frf$ is strictly concave on $\bbR\bs\frf$, there are no critical points in the two unbounded components of $\bbR\bs\frf$ and precisely one in each finite component.  If now $x_0\in\bbR$ is in the complement of a compact set $\fre\subset\bbR$, this translates into a similar property for the critical points of $g_\fre(\dott,x_0)$.  Namely, $\calC\subset(\hatR)\bs\fre$ and each gap of $\fre$, except the gap containing $x_0$ together with $(-\infty,x_-)\cup (x_+,\infty)\cup\{\infty\}$, contains exactly one point of $\calC$ and $g_\fre(\dott,x_0)$ attains its maximum in the gap exactly at this point.

As a second application of \eqref{3.9}, we now derive the Totik--Widom upper bound for the Widom factors of the residual polynomials:

\begin{theorem} \lb{T3.10} Let $\fre\subset\bbR$ be a regular, not connected, compact PW set and let $x_0\in\bbR\bs\fre$.  Then
\begin{align} \label{3.15}
  W_n(\fre,x_0) < 2\exp[\PW(\fre,x_0)], \quad n\in\bbN
\end{align}
Moreover, equality holds in the limit for a subsequence $n_j\to\infty$ if and only if, for any gap of $\fre$, $P_{n_j}$ has a zero in that gap for $j$ large which approaches the unique critical point of $g_\fre(\dott,x_0)$ as $j\to\infty$ and the component of $\fre_{n_j}$ in that gap shrinks to that critical point as $j\to\infty$.
\end{theorem}

\begin{remarks}  1. It can be shown that whenever the zeros in some gap converge to a limit in that gap, the corresponding component of $\fre_{n_j}$ in that gap shrinks to that point exponentially fast, so, in the final assertion, we could drop the last clause.

2. If $\fre$ is an interval, then for all $n$, $W_n(x_0,\fre)=2$ (by \eqref{3.9}), so since $\PW(\fre,x_0)=0$, we have equality in \eqref{3.15}, which is why we have the condition ``not connected''.
\end{remarks}

\begin{proof} The proof follows that of the analog we used in the Chebyshev case in \cite{CSZ1}.  As there, we start by recalling why if $\frf$ is a non-polar, compact subset of $\bbC$, the equilibrium measure, $d\rho_\frf$, is also called harmonic measure.  For one can show (see Conway \cite{Conway} or Simon \cite[Corollary~3.6.28]{HA}) that if $f$ is a continuous function on $\frf$, there is a unique function, $u_f$, harmonic on $(\bbC \cup \{\infty \})\bs\frf$, which approaches $f(x)$ for q.e. $x \in \frf$ (i.e., solves the Dirichlet problem) and so that
\begin{equation} \lb{3.16}
   u_f(\infty) = \int_{\frf}f(x)d\rho_{\frf}(x)
\end{equation}
The function $u_f(z)=g_\fre(z,x_0)-g_n(z,x_0)$ is harmonic on $(\bbC \cup \{\infty \})\bs\fre_n$ (because the logarithmic singularities at $x=x_0$ cancel and lead to a removable singular point).  By \eqref{3.12A}, we have that $u_f(\infty)=g_\fre(x_0)-g_n(x_0)$.  Moreover, the limiting value of $u_f$ on $\fre$ is 0 (since $\fre\subset\fre_n$) and is $g_\fre(x,x_0)$ for $x\in\fre_n\bs\fre$.  Therefore, by \eqref{3.16} with $\frf=\fre_n$,
\begin{equation}\label{3.18}
  g_\fre(x_0)-g_n(x_0) = \int_{\fre_n\bs\fre} g_\fre(x,x_0)\, d\rho_n(x)
\end{equation}
By Proposition~\ref{P3.7}~(a), $\rho_n$ of each gap of $\fre$ is at most $1/d_n$.

First, suppose $d_n=n$.  Then, summing over gaps in \eqref{3.18} and using the fact that the maximum of $g_\fre(x,x_0)$ in a gap, $G_j$, is at the critical point, $c_j$, we have
\begin{equation} \lb{3.19}
  g_\fre(x_0)-g_n(x_0) \le \frac{1}{d_n} \sum_{c_j\in\calC} g_\fre(c_j,x_0) =  \frac{1}{d_n} \PW(\fre,x_0)
\end{equation}
Then using \eqref{3.9} and the fact that
\begin{equation}
 g_n(x_0) \le g_\fre(x_0) \; \Rightarrow \; {\bigl(1+e^{-2d_n g_n(x_0)}\bigr)^{-1} \le \bigl(1+e^{-2d_n g_\fre(x_0)}\bigr)^{-1}}
\end{equation}
we get
\begin{align}
  \norm{R_{x_0,n}}_\fre &= \frac{2}{e^{d_n g_n(x_0)}+e^{-d_n g_n(x_0)}}
  = \frac{2e^{-d_n g_n(x_0)}}{1+e^{-2d_n g_n(x_0)}} \nonumber \\
  &\le \frac{2e^{-d_n g_{\fre}(x_0)}e^{\PW(\fre,x_0)}}{1+e^{-2d_n g_{\fre}(x_0)}}
  = \frac{2e^{\PW(\fre,x_0)}}{e^{d_n g_{\fre}(x_0)}+e^{-d_n g_{\fre}(x_0)}} \lb{3.20}
\end{align}
By \eqref{1.4}, this is the desired bound \eqref{3.15} in the case $d_n=n$.

Next, suppose $d_n=n-1$.  In this case, by Proposition~\ref{P3.7} (b), $\fre_n$ does not intersect the ``unbounded gap'' of $\fre$ (i.e., $(-\infty,x_-)\cup (x_+,\infty)\cup\{\infty\}$).  Let $\calC'=\calC\bs\{c_{\infty}\}$, where $c_{\infty}$ is the critical/maximum point of $g_\fre(\dott,x_0)$ in the unbounded gap.  Then
\begin{equation} \lb{3.21}
  g_\fre(x_0)-g_n(x_0) = \int_{\fre_n\bs\fre} g_\fre(x,x_0)\, d\rho_n(x)
  \le \frac{1}{d_n} \sum_{c_j\in\calC'} g_\fre(c_j,x_0)
\end{equation}
and since, by \eqref{3.12A}, $g_\fre(x_0)=g_\fre(\infty,x_0) \le g_\fre(c_\infty,x_0)$, we have
\begin{equation} \lb{3.22}
   \Bigl(1+\frac{1}{d_n}\Bigr)g_\fre(x_0)-g_n(x_0)
      \le \frac{1}{d_n} \sum_{c_j\in\calC}g_\fre(c_j,x_0)
      = \frac{1}{d_n} \PW(\fre,x_0)
\end{equation}
So using \eqref{3.9} and the fact that $d_n g_n(x_0) \le (d_n+1)g_\fre(x_0)$ implies
\begin{equation}
 {\bigl(1+e^{-2d_n g_n(x_0)}\bigr)^{-1} \le \bigl(1+e^{-2(d_n+1)g_\fre(x_0)}\bigr)^{-1}}
\end{equation}
we get
\begin{align}
  \norm{R_{x_0,n}}_\fre &= \frac{2}{e^{d_n g_n(x_0)}+e^{-d_n g_n(x_0)}}
  = \frac{2e^{-d_n g_n(x_0)}}{1+e^{-2d_n g_n(x_0)}} \lb{3.23} \\
  &\le \frac{2e^{-(d_n+1)g_{\fre}(x_0)}e^{\PW(\fre,x_0)}}
            {1+e^{-2(d_n+1) g_{\fre}(x_0)}}
  = \frac{2e^{\PW(\fre,x_0)}}{e^{(d_n+1)g_{\fre}(x_0)}+e^{-(d_n+1)g_{\fre}(x_0)}} \nonumber
\end{align}
By \eqref{1.4}, this is the desired bound \eqref{3.15} in the case $d_n=n-1$.

Because $n_j\to\infty$, asymptotic equality in \eqref{3.15} is equivalent to $\lim_{j\to\infty} r_{x_0, n_j}e^{n_jg_\fre(x_0)} = 2 \exp\left[\PW(\fre,x_0)\right]$ which, by \eqref{3.9}, is equivalent to
$\lim_{j\to\infty} d_{n_j}[g_\fre(x_0)-g_{n_j}(x_0)] = \PW(\fre,x_0)$.  By \eqref{3.18} and dominated convergence for sums, this is equivalent to knowing that for each fixed gap, $G$, of $\bbR\bs\fre$, one has that $d_{n_j}\rho_{n_j}(G\cap \fre_{n_j}) = 1$ and that $G\cap \fre_{n_j}$ is more and more concentrated about $c$, the critical point that lies in $G$.
\end{proof}

We close this section with a sufficient condition for saturation of the lower bound in the limit which is complementary to the final result in the last theorem.

\begin{theorem} \lb{T3.11} Equality in \eqref{3.11} holds in the limit for a subsequence $n_j\to\infty$ if, for any gap, $G$, of $\fre$, one has that $\max_{y\in\fre_{n_j}\cap G} \dist(y,\fre)\to 0$.
\end{theorem}

\begin{remarks} 1. With a little more effort, it should be possible to prove the stated sufficient condition is necessary.  The point is that we expect if in the limit, $\fre_{n_j}\cap G$ contains some point, $y$, then as $j\to\infty$, $\fre_{n_j}\cap G$ contains an entire band concentrated about $y$ which contributes $g_\fre(y,x_0)>0$ to $d_{n_j}$ times the integral on the right of \eqref{3.18}.

   2.  The sufficient condition of this theorem for a gap, $G$, is equivalent to one that says for large $n_j$, any zero of $P_{n_j}$ in $G$ must approach the edges of $G$.
\end{remarks}

\begin{proof}  As in the last paragraph of the last proof, one has asymptotic equality in \eqref{3.11} if and only if $\lim_{j\to\infty} d_{n_j}[g_\fre(x_0)-g_{n_j}(x_0)] = 0$.  By the dominated converge theorem for sums and \eqref{3.18}, this happens if and only if for any gap, $G_k$, one has that
\begin{equation}\label{3.24}
  \lim_{j\to\infty} \int_{\fre_{n_j}\cap G_k} g_\fre(x,x_0)\, d\rho_{n_j}(x) = 0
\end{equation}
Since $g_\fre(\cdot,x_0)\to 0$ at the edges and $d_n\rho_n(G_k)$ is at most $1$, this follows from the assumption that $\max_{y\in\fre_{n_j}\cap G_k} \dist(y,\fre)\to 0$.
\end{proof}

\section{Some Examples} \lb{s4}

This section will discuss some illuminating examples, the first two (as well as Example~\ref{E2.3} on lemniscates) deal with the complex case and the others with the real case that has been our main focus here.  In discussing some asymptotics, the classical Chebyshev polynomials of the first kind, which we denote by $C_n$, will be useful:
\begin{equation}\label{4.1}
  C_n(\cos(\theta)) \equiv \cos(n\theta)
\end{equation}
Despite the name, they are not Chebyshev polynomials since they are not monic although they are multiples of and, indeed, dual Chebyshev polynomials for $[-1,1]$.  We note their asymptotics which is classical (but also a special case of the results in \cite{CSZ1, CSYZ2})
\begin{equation}\label{4.2}
  C_n(z) \sim \frac{1}{2} \left[\frac{z+\sqrt{z^2-1}}{2}\right]^n
\end{equation}

\begin{example} [Cases with $n$-fold symmetry and $\deg(R_{z_0,j})=0$] \lb{E4.1} If $z_0$ is in a bounded component of $\bbC\bs\fre$, then, by the maximum principle, any polynomial, $P$, with $P(z_0)=1$ has $\norm{P}_\fre \ge 1$ with strict inequality if $P$ is not constant.  It follows that $R_{z_0,n}(z)\equiv 1$ for all $n$ so to get non-trivial results, we should only consider cases with $z_0$ in the unbounded component.  We consider $z_0=0$.  If $\fre$ is invariant under rotation by angle $2\pi/n$ about $0$, then uniqueness of $R$ implies that $R_{z_0=0,j}(e^{2\pi/j}z)=R_{z_0=0,j}(z)$ for all $j$, so the only terms that are allowed in $R$ are of the form $c_{kn}z^{kn};\,k=0,1,\ldots$.  It follows that $R_{z_0=0,j}(z)=1$ for $0\le j\le n-1$.  It can even happen that this holds for $j=n$ (and so for $j=n+1,\ldots,2n-1$) for suppose that $1,e^{i\pi/n}\in\fre$.  Noticing that for any $c\in\bbC$, one has that $|1+c|^2+|1-c|^2=2(1+|c|^2)$, we see that, for $c\ne 0$, if $P(z)=1+cz^n$, either $|P(1)|>1$ or $|P(e^{i\pi/n})|>1$ and thus we also have that $R_{z_0=0,n}(z)=1$.  It is easy to find proper, closed, perfect subsets of $\partial\bbD$ which are $n$-fold invariant containing all the $2n$-th roots of unity.  In this way, one can construct, for any finite $m$, $n$ fold invariant sets with $R_{z_0=0,j}(z)=1$ for $0\le j\le mn-1$.  However, we note that, by Theorem~\ref{T2.4} (a), if $z_0$  is in the unbounded component, for any fixed $\fre$, we have that $\deg(R^{(\fre)}_{z_0,j})\to\infty$ as $j\to\infty$.
\end{example}

\begin{example}  [{$\fre=[-1,1];\,z_0\notin\bbR$}] \lb{E4.2} Let $\fre= [-1,1]$.  As we've seen, if $x_0\in\bbR\bs\fre$, then
\begin{equation}\label{4.2A}
  R_{x_0,n}(x) = C_n(x)/C_n(x_0); \quad \; r_{x_0,n}=1/|C_n(x_0)|
\end{equation}
This is so simple it is natural to guess, or at least hope, that it extends to complex $z_0$.  But it does not.  The dual residual problem was solved when $z_0$ is on the imaginary axis by Freund--Ruscheweyh \cite{FR86} and for general $z_0\in\bbC\bs\fre$ by Yuditskii \cite{Yud99}.  The formula for $r_{z_0,n}$ is quite complicated using elliptic functions.  Here, we only make a few remarks.  We first note that $1/|C_n(z_0)|$ diverges as $z_0$ approaches a zero on the real axis while, of course, $r_{z_0,n}\le 1$.  Indeed, in the entire lemniscate $\{z : |C_n(z)|<1\}$, one has that $r_{z_0,n}<1/|C_n(z_0)|$.  But much more is true.  Take $n=1$ and let $P_\eps(z)=(z+i\eps)/(z_0+i\eps)$ when $\Im(z_0)>0$ and $\eps>0$. Then
\begin{equation}\label{4.2B}
  \norm{P_\eps}_\fre^2 = \frac{1+\eps^2}{|z_0|^2+\eps^2+2\Im(z_0)\eps} < \frac{1}{|z_0|^2}
\end{equation}
for $\eps$ small.  This shows that everywhere off $\bbR$, one has that $r_{z_0,n=1} < 1/|C_{n=1}(z_0)|$ and Yuditskii's work implies the analog for all $n$.
\end{example}
It is natural to also consider the residual polynomials of an ellipse with foci on the real line.  In this case the polynomials $C_n(z)/C_n(z_0)$ may or may not be the residual polynomials even when the point $z_0$ is real and outside the ellipse.  It depends on the configuration; see \cite{FiFr90,FiFr91} for further details.
\begin{example} [$n=1$, $x_0\in\bbR$] \lb{E4.3}  There is no maximum principle for general polynomials on $\bbR$ but there is for affine functions, which unless they are constant take their maximum over a bounded closed interval at an endpoint.  It follows that if $\fre\subset\bbR$ and $x_0$ is contained in a real gap of $\bbR\bs\fre$, then $R_{x_0,n=1}(z)\equiv 1$ showing that it can very often happen that $\deg(R_{x_0,n})=n-1$.  Of course, $R_{x_0,n}\equiv 1$ can only happen for $n=1,0$ by Theorem~\ref{3.4} (b).
\end{example}

\begin{example} [$\fre$ reflection invariant about $x_0=0$] \lb{E4.4} The only real analog of the $n$-fold symmetry of Example~\ref{E4.1} is $2$-fold symmetry of sets $\fre\subset\bbR$ with the property that $x\in\fre\Rightarrow -x\in\fre$.  We also suppose $0\notin\fre$ and take $x_0=0$.  As in Example~\ref{E4.1}, the $R_{x_0=0, n}$ are even polynomials and thus $\deg(R_{x_0=0, 2n+1})=2n$. With $T_n$ the Chebyshev polynomials, it therefore follows from Theorem~\ref{T3.6} that
\begin{equation}\label{4.3}
  R_{x_0=0,n}(x) = \left\{
                     \begin{array}{ll}
                       T_{n-1}(x)/T_{n-1}(0), & \hbox{ if $n$ is odd} \\
                       T_{n}(x)/T_{n}(0), & \hbox{ if $n$ is even}
                     \end{array}
                  \right.
\end{equation}
It is interesting to see how existence of the limits in \eqref{1.7} and \eqref{1.8} for some sequence of even $n_j\to\infty$ implies the existence of limits for $n_j+1\to\infty$ and what the relation has to be of the limit $F_n$ to the limiting $F_{n+1}$.
\end{example}

\begin{example}  [{$\fre=[-b,-a]\cup[a,b]$ for $0<a<b$, $x_0=0$}] \lb{E4.5} This special case of Example~\ref{E4.4} has explicit formulae.  Let $q$ be the quadratic polynomial
\begin{equation}\label{4.4}
  q(z) = \frac{2(z^2-a^2)-(b^2-a^2)}{b^2-a^2}
\end{equation}
picked so that $q(\pm a)=-1,\,q(\pm b)=1$ which implies that $\fre=q^{-1}([-1,1])$ and that $g_\fre(z)=\tfrac{1}{2}g_{[-1,1]}(q(z))$ which in turn implies that
\begin{equation}\label{4.5}
  B_\fre(z)^2 = B_{[-1,1]}\bigl(q(z)\bigr)
\end{equation}
This function is thus single-valued, so $\chi_\fre^2\equiv 1$, indeed $\chi_\fre(\gamma) = (-1)^{\#(\gamma)}$ where $\#(\gamma)$ is the number of times that $\gamma$ winds around $[-a,-b]$ plus the number of times it winds around $[a,b]$.  The Widom minimizer for $\chi_\fre^{2n}$ is thus $1$ and \eqref{1.8} implies that
\begin{equation}\label{4.6}
  \lim_{n\to\infty} e^{2ng_\fre(x_0=0)}B_\fre(z)^{2n}R_{x_0=0,2n}(z) = 1
\end{equation}
for all $z\in\bbC\bs\fre$. Since $T^{(\fre)}_{2k}(z)=T^{[-1,1]}_k(q(z))$ and $R_{x_0=0,2n}$ is given by \eqref{4.3}, this is consistent with \eqref{4.2}. For $n=2k+1$, in the language of the next section, $R_{x_0=0, n}$ has a zero at infinity, so the corresponding $Q_{\chi_\fre}$ is a $B_S$ where $S$ has a zero at $\infty$ and none in gap $(-a,a)$, that is, $Q_{\chi_\fre}=B_\fre$ which also has a zero at infinity. By \eqref{5.4}, we see that $F_{\chi_\fre}=e^{g_\fre(x_0=0)}B_\fre$. Noting that $R_{x_0=0,2k+1}=R_{x_0=0,2k}$, we see that
\begin{equation}
e^{(2k+1)g_\fre(x_0=0)}(B_\fre)^{2k+1}R_{x_0=0,2k+1} = F_{\chi_\fre}e^{2kg_\fre(x_0=0)}(B_\fre)^{2k}R_{x_0=0,2k}
\end{equation}
which, by \eqref{4.6}, converges to $F_{\chi_\fre}$ consistent with \eqref{1.8}.
\end{example}

\begin{example} [$\fre$ a period-$n$ set] \lb{E4.6} Suppose $\fre$ is a period-$n$ set, i.e., there is a polynomial, $\Upsilon$, of degree $n$ exactly which is a dual residual or Chebyshev polynomial of degree exactly $n$ so that $\norm{\Upsilon}_\fre=1$ and $\fre=\Upsilon^{-1}([-1,1])$.  The maxima and minima of $\Upsilon$ all lie in $\fre$ and, by the analysis in the proof of Proposition \ref{P3.7} (b), they occur with singlets at the ends and pairs in the middle.  Take $x_0$ in one of the bounded gaps, $G$, of $\bbR\bs\fre$.  Then by taking both endpoints of $G$, the two singlets in the extreme points, and one from each of the other pairs of extreme points, we obtain an $x_0$-alternating set with $n+2$ points.  It follows from Theorem~\ref{T3.2} that for $j\geq 1$,
\begin{equation}\label{4.7}
  R_{x_0,m}(x) = C_j\bigl(\Upsilon(x)\bigr)/C_j\bigl(\Upsilon(x_0)\bigr)
\end{equation}
for $m=jn$ or $m=jn+1$. As in the analysis in the last example, one has that $F_{jn}=1$ and $F_{jn+1} = e^{g_\fre(x_0)}B_\fre$.
\end{example}

\begin{example} [$\fre$ a period-$n$ set shrunk at one end] \lb{E4.7} Start out with a period-$n$ set, $\frf$, with $y_\pm$ the top/bottom of the set and $\Upsilon$ as defined in the last example.  Let $B=[a,y_+]$ be the top connected component of $\frf$ and consider $\fre=\frf\bs (c,y_+]$, where $a<c<y_+$.  Suppose $x_0$ is in one of the bounded gaps of $\fre$.  Take the $x_0$-alternating set with $n+2$ points for $R_{x_0,n}^{(\frf)}$ as in the previous example and remove $y_+$ from it.  We then get an $x_0$-alternating set in $\fre$ with $n+1$ points showing that  $R_{x_0,n}^{(\fre)}=\Upsilon(x)/\Upsilon(x_0)$.  One interesting feature of this example is that depending on whether we pick $c$ above or below the zero of $\Upsilon$ in $[a,y_+]$, we see that the extra zero not accounted for in Theorem~\ref{T3.4}~(a) can either lie in $[x_-,x_+]$ or not.
Moreover, if $a$ is not an endpoint of the gap containing $x_0$, then we can even take $\fre=\frf\bs B$ showing that $\fre_n=\frf$ may contain an extra component outside $[x_-,x_+]$.
Similarly, if we shrink one of the internal bands of $\frf$ we get an example where $\fre_n$ lies within $[x_-,x_+]$ but $\fre\ne\fre_n$.
\end{example}

\begin{example} [{Example where the limit points of $W_n(\fre,x_0)$ fill the interval $\bigl[2,2\exp[\PW(\fre,x_0)]\bigr]$}] \lb{E4.8}
\eqref{1.5} sets upper and lower bounds on Widom factors and so on their possible limit points.  In this example, we want to discuss finite gap sets where the set of limit points is the whole interval $\bigl[2,2\exp[\PW(\fre,x_0)]\bigr]$.  These are just analogs of what we discussed for Chebyshev polynomials in \cite{CSZ3}.  We'll need the notions of gap sets and ideas from the next section.  As in \cite{CSYZ2, CSZ3}, if $\fre$ is a finite gap set with $m$ connected components so that no $m-1$ of them have a linear rational relation among their harmonic measures, then for $x_0$ in a bounded component of $\bbR\bs\fre$, and any gap set whose gap collection doesn't include the gap with $x_0$, there is a sequence $R_{x_0,n_j};\, j=1,2,\ldots$ whose zeros inside the gaps approach exactly the points of the gap set.  As in those papers, one can show for $K_k\in\calG_0$, that for $j$ large, we have that $\fre_{n_j}\cap K_k$ is a closed interval entirely within $K_k$ that shrinks to the point $x_k$, so
\begin{equation}\label{4.8}
  n_j\int_{K_k} g_\fre(x,x_0) d\rho_{n_j}(x) \to g_\fre(x_k,x_0)
\end{equation}
For $K_k\notin\calG_0$, one can show that $\fre_{n_j}\cap K_k$ shrinks to the edges of the gap, so since $g_\fre(\cdot,x_0)$ vanishes there, the integral goes to zero.  Using \eqref{3.21} and \eqref{3.9}, we see that
\begin{equation}\label{4.9}
  \lim_{j\to\infty} W_{n_j}(\fre,x_0) = 2\sum_{K_k\in\calG_0} g_\fre(x_k,x_0)
\end{equation}
It is then easy to see, knowing that all possible gap sets occur, that the set of limits is the entire interval $\bigl[2,2\exp[\PW(\fre,x_0)]\bigr]$.
\end{example}

\section{Szeg\H{o}--Widom Asymptotics} \lb{s5}

The purpose of this section is to obtain fairly explicit almost periodic asymptotics for $R_{x_0,n}$ as $n\to\infty$. Throughout this section, we assume that $\fre\subset\bbR$ is a regular Parreau--Widom set with DCT (discussed below) and let $\Omega$ be its complement in the Riemann sphere, that is,
\begin{equation} \lb{5.1}
  \Omega = (\hatC) \bs \fre
\end{equation}
Under these hypotheses, \cite{CSYZ2} proved explicit asymptotics for Chebyshev polynomials (earlier \cite{CSZ1} had proven this for finite gap sets).  Not only will our proof here have a lot in common with the proof in \cite{CSYZ2}, it will be able to use some parts of that proof verbatim.

Let $\wti{\Omega}$ be the universal cover of $\Omega$ with $\cz:\wti{\Omega}\mapsto\Omega$ the covering map.  We will be interested in analytic functions $f:\wti{\Omega}\mapsto\bbC$ so that there is a single-valued function $g:\Omega\mapsto\bbC$ with $g(\cz(w))=|f(w)|$.  Given such a function, by the monodromy theorem \cite[Theorem~11.2.1]{BCA}, if $\pi_1(\Omega,x_0)$ is the fundamental group, there is a map $\chi:\pi_1(\Omega,x_0)\mapsto\partial\bbD$ so that if $\gamma$ is a curve in $\wti{\Omega}$ with $\cz(\gamma(0))=\cz(\gamma(1))=x_0$, then $f(\gamma(1))=\chi([\cz\circ\gamma])f(\gamma(0))$, that is, $\chi$ describes the phase change under continuing the mutivalued projection of $f$ around a closed curve in $\Omega$.  It is easy to see that $\chi$ is a character.  We'll call $f$ a \emph{character automorphic function}, or a $\chi$-automorphic function when we want the character to be explicit.  By construction, $B_\fre$ normalized by $B_\fre(x_0)>0$ (or rather its  single-valued lift to $\wti{\Omega}$) is character automorphic.  We use $\chi_\fre$ for the associated character.

Given a character, $\chi$, of $\pi_1(\Omega,x_0)$, we define the \emph{Widom minimizer}, $F_\chi(z)$, as a bounded $\chi$-automorphic function with $F_\chi(x_0)=1$ and
\begin{equation} \lb{5.2}
  \norm{F_\chi}_\Omega = \inf\{\norm{h}_\Omega \,:\, h \in H^\infty(\Omega,\chi),\; h(x_0) = 1\}
\end{equation}
The \emph{dual Widom maximizer}, $Q_\chi$, is a $\chi$-automorphic function with $\norm{Q_\chi}_\Omega=1$ and
\begin{equation} \lb{5.3}
  Q_\chi(x_0) = \sup\{h(x_0): h\in H^\infty(\Omega,\chi), \; \norm{h}_\Omega = 1, \; h(x_0) > 0\}
\end{equation}
It is easy to see that
\begin{equation} \lb{5.4}
  Q_\chi = F_\chi/\norm{F_\chi}_\Omega, \quad F_\chi=Q_\chi/Q_\chi(x_0), \quad \norm{F_\chi}_\Omega = 1/Q_\chi(x_0)
\end{equation}

In the case $x_0=\infty$, these extremal functions, which we'll call $Q_\chi^\infty$, have been studied extensively.  If the PW property holds, for our finite $x_0$ situation, the dual Widom maximizer $Q_\chi^\infty$ exists and is unique; see, for example, \cite{CSYZ2}.  We will use $F_n$ as shorthand for $F_{\chi_\fre^n}$.

A final element we need is the notion of the Direct Cauchy Theorem (DCT) property.  There are many equivalent definitions of DCT -- see Hasumi \cite[pg. 151]{Hasu} or Volberg--Yuditskii \cite{VY}.  Rather than stating a formal definition, we quote a theorem that could be used as one definition of DCT: $\fre$ has the DCT property if and only if $Q_\chi^\infty(\infty)$ depends continuously on $\chi$.

We note that any homogeneous subset of $\bbR$ (in the sense of Carleson) obeys DCT \cite{SY}.  On the other hand, Hasumi \cite{Hasu} has found rather simple explicit examples (with thin components) of subsets of $\bbR$ which obey PW but not DCT.  Volberg--Yuditskii \cite{VY} have even found examples all of whose reflectionless measures are absolutely continuous.

By a conformal transformation, $\Upsilon$, the set and any fixed point $x_0\in\bbR\bs\fre$ can be mapped to $\frf=\Upsilon[\fre]$ and $\infty=\Upsilon(x_0)$. By \cite[pg. 177]{Hasu}, $\frf$ has DCT if and only if $\fre$ does.  It follows that for our problem of finite $x_0$ and $Q_\chi$, one has first that $\chi\mapsto Q_\chi(x_0)$ is continuous and then, as in \cite{CSYZ2}, that for each $z\in\wti{\Omega}$, $Q_\chi(z)$ and $F_\chi(z)$ as well as $\norm{F_\chi}$ are continuous in $\chi$.  Thus $F_n(z)$ and $\norm{F_n}$ are almost periodic functions of $n$.

With this background, we can turn to the proof of Theorem~\ref{T1.1}.  We note it is easy to see, as mentioned earlier, that uniform convergence on compact subsets of $\Omega$ with cuts implies uniform convergence on compact subsets of $\wti{\Omega}$.

As in \cite{CSYZ2}, we begin by discussing the function defined on $\Omega_n\equiv(\hatC)\bs\fre_n$ by
\begin{equation} \label{5.5}
   L_n(z) \equiv B_\fre(z)^n\Delta_n(z)=B_\fre(z)^n B_n(z)^{d_n}+B_\fre(z)^n/B_n(z)^{d_n}
\end{equation}
Since $|B_\fre|<1$ on $\Omega$, it suffices to consider the asymptotics of
\begin{equation} \lb{5.6}
   M_n(z) \equiv B_\fre(z)^n/B_n(z)^{d_n}
\end{equation}
which by \eqref{3.12} is bounded in magnitude by $1$ on $\Omega_n$.

To control the convergence to an almost periodic orbit, we will control limits along enough subsequences.  The complement of $\fre$ in $\hatR$ is a disjoint union of bounded open components and an infinite component which is $(x_+,\infty)\cup\{\infty\}\cup(-\infty,x_-)$.  We'll call these components the gaps and denote the set of gaps by $\calG$.  A \emph{gap collection} is a subset $\calG_0 \subset \calG$.  A \emph{gap set} is a gap collection, $\calG_0$, and for each $K_k \in \calG_0$, a point $x_k \in K_k$.  For any bounded gap $K = (\alpha,\beta)$, we define
\begin{equation} \lb{5.7}
  K^{(\eps)} = \left(\alpha+\eps|\tfrac{\beta-\alpha}{2}|,\beta-\eps|\tfrac{\beta-\alpha}{2}|\right),   \quad \eps \in (0,1)
\end{equation}
and, for the unbounded gap, $K^{(\eps)} = (x_+ +\eps,\infty)\cup\{\infty\}\cup(-\infty,x_- -\eps)$.

For each gap set $S$, we define the associated Blaschke product
\begin{equation} \lb{5.8}
  B_S(z) = \prod_{K_k \in \calG_0} B_\fre(z,x_k)
\end{equation}
where we normalize all the Blaschke functions by $B_\fre(x_0,x_k)>0$.
If $\fre$ is DCT then, by \cite{CSYZ2}, we know that each such $B_S$ is a dual Widom maximizer $Q_\chi$. If $x_k\in K_k$ and $c_k$ is the critical point of $g_\fre(\cdot,x_0)$ in $K_k$, then $|B_\fre(x_0,x_k)|=\exp(-g_\fre(x_0,x_k))=\exp(-g_\fre(x_k,x_0)\ge\exp(-g_\fre(c_j,x_0))$, so we have that $Q_\chi(x_0)=B_S(x_0)$ is bounded away from zero uniformly in $S$,
\begin{equation} \label{5.9}
  |B_S(x_0)|\ge\exp[-\PW(\fre,x_0)]
\end{equation}
with equality occurring when the gap set $S$ consists of all the gaps $\calG$ and in each gap $K_k\in\calG$ the point $x_k$ is the critical point of $g_\fre(\dott,x_0)$.

In the next theorem we will think of $R_{x_0,n}$ with $d_n=n-1$ as a degenerate polynomial of degree $n$ with a zero at infinity.

\begin{theorem} \lb{T5.1}
Let $n_j \to \infty$ so that for some gap set $S$ we have that if $K_k \in \calG_0$, then for large $j$, $R_{x_0,n_j}(z)$ has a zero $x_j^{(k)}$ in $K_k$ which converges to $x_k$ as $j \to \infty$ and so that for any $K\in\calG\bs\calG_0$, and for all $\eps\in(0,1)$, $R_{x_0,n_j}(z)$ has no zero in $K^{(\eps)}$ for all large $j$.  Then, as $j\to\infty$, $M_{n_j}(z) \to B_S(z)$ uniformly on compact subsets of $\Omega\bs\{x_k\}_{K_k\in\calG_0}$.
\end{theorem}

\begin{proof} This result is similar to \cite[Theorem~4.1]{CSYZ2}.  The argument given in \cite[Section~4]{CSYZ2} needs only a slight modification when the gap set $S$ contains the infinite gap $K_0=(x_+,\infty)\cup\{\infty\}\cup(-\infty,x_-)$.  To deal with the infinite gap, we need to consider two subcases:

 (1) $d_{n_j}=n_j-1$. In this case, $R_{x_0,n_j}$ is a constant multiple of the Chebyshev polynomial $T_{n_j-1}$. Hence, by \cite[Theorem~4.1]{CSYZ2},
\begin{equation} \lb{5.10}
  [B_\fre/B_{n_j}]^{n_j-1} \to B_{S'}
\end{equation}
where $S'$ is the gap set $S$ with the infinite gap removed, and so
\begin{equation} \lb{5.11}
  M_{n_j}=B_\fre^{n_j}/B_{n_j}^{n_j-1} = B_\fre[B_\fre/B_{n_j}]^{n_j-1} \to B_\fre B_{S'}=B_{S}
\end{equation}

(2) $d_{n_j}=n_j$. In this case, it is possible that the band $\fre_{n_j}\cap K_0$ will not shrink in size as $j\to\infty$. However, as we shall explain, its endpoints must still converge to $x_0=\infty$. We know that $x_j^{(0)}\to x_0=\infty$ and, by \eqref{2.7} and \eqref{2.8},
\begin{equation} \lb{5.12}
  \left|\frac{R_{x_0,n}(z)}{\norm{R_{x_0,n}}_\fre}\right|^{1/n} \to \exp[g_\fre(z)]
\end{equation}
uniformly on compact sets not containing zeros of $R_{x_0,n}(z)$.  Therefore, on each compact subset of $K_0$, we have $|R_{x_0,n}|>\norm{R_{x_0,n}}_\fre$ for large $n$. This implies that the endpoints of $\fre_{n_j}\cap K_0$ converge to $x_0=\infty$ and that is what we need for the argument of \cite[Section~4]{CSYZ2}. Alternatively, we could use the conformal transformation $z\mapsto f(z)=(z-x_0)^{-1}$ to define the sets $\frf=f(\fre)$ and $\frf_n=f(\fre_n)$. Then for each gap $f(K_k)$ of $\frf$, the band $\frf_{n_j}\cap f(K_k)$ shrinks to $f(x_k)$ which is the setting of \cite[Section~4]{CSYZ2} and hence the result follows from \cite[Section~4]{CSYZ2}.
\end{proof}

\begin{proof} [Proof of Theorem~\ref{T1.1}]  By \eqref{5.5}--\eqref{5.6}, our previous remark that each $B_S$ is a dual Widom maximizer, and the previous theorem, for all $z\in\Omega$,
\begin{equation} \label{5.13}
  \lim_{n\to\infty} L_n(z)-Q_n(z) = \lim_{n\to\infty} M_n(z)-Q_n(z) = 0
\end{equation}
At $z=x_0$, this yields
\begin{equation}
  \lim_{n\to\infty} L_n(x_0)/Q_n(x_0) = 1
\end{equation}
since $1/Q_n(x_0)$ is bounded above by \eqref{5.9}.  Recalling \eqref{5.5} and that $B_\fre(x_0)=\exp(-g_\fre(x_0))$ and, by \eqref{3.5}, $\Delta_n(x_0) = 2/\norm{R_{x_0,n}}_\fre$ then shows
\begin{equation} \label{5.14}
  \lim_{n\to\infty} e^{n g_\fre(x_0)} \norm{R_{x_0,n}}_\fre Q_n(x_0) = 2
\end{equation}
This implies \eqref{1.7} since the sequence $\norm{F_n}_\Omega=1/Q_n(x_0)$ is bounded above by \eqref{5.9}.

By \eqref{5.5} and \eqref{3.5},
\begin{equation} \lb{5.15}
  e^{n g_\fre(x_0)} B_\fre(z)^n R_{x_0,n}(z)   = \tfrac12 e^{n g_\fre(x_0)} \norm{R_{x_0,n}}_\fre L_n(z)
\end{equation}
Since, by \eqref{5.4}, $Q_n(z)=Q_n(x_0)F_n(z)$ and $e^{n g_\fre(x_0)} \norm{R_{x_0,n}}_\fre$ is a bounded sequence by \eqref{1.7}, it follows from \eqref{5.13} and \eqref{5.14} that
\begin{align}
  0 &= \lim_{n\to\infty} \tfrac12 e^{n g_\fre(x_0)} \norm{R_{x_0,n}}_\fre  \left[ L_n(z)-Q_n(z) \right] \nonumber \\
    & = \lim_{n\to\infty} \left[e^{n g_\fre(x_0)} B_\fre(z)^n R_{x_0,n}(z) - \tfrac12 e^{n g_\fre(x_0)} \norm{R_{x_0,n}}_\fre Q_n(x_0) F_n(z) \right] \nonumber \\
    & = \lim_{n\to\infty} \left[e^{n g_\fre(x_0)} B_\fre(z)^n R_{x_0,n}(z) - F_n(z) \right] \lb{5.16}
\end{align}
which is \eqref{1.8}.
\end{proof}

In \cite{Tot11} and \cite{Tot14}, Totik studied the $\liminf$ of Widom factors for Chebyshev polynomials and when the limit exists.  We want to show it is easy to prove and extend (both to residual polynomials and in the case of Theorem \ref{T5.2} to a larger family of sets) these results using the ideas of this section.

\begin{theorem} \lb{T5.2} Let $\fre\subset\bbR$ be a compact PW set with DCT and let $x_0\in\bbR\bs\fre$.  Then $\liminf_{n\to\infty} W_n(\fre,x_0)=2$.  This holds also for $x_0=\infty$ if $W_n(\fre,\infty)$ is interpreted as $t_n/C(\fre)^n$.
\end{theorem}

\begin{remark} Totik \cite[Theorem~3]{Tot11} has this result for finite gap sets in the Chebyshev case; indeed, he has some control on the rate of convergence of $\inf_{j\le n} W_{j}(\fre,\infty)$ to $2$.
\end{remark}

\begin{proof} We consider the case $x_0\ne\infty$ (the Chebyshev case follows using the arguments in \cite{CSYZ2, CSZ3} in place of the ideas of this section).  If $G$ is a finite or infinite dimensional torus, it is easy to see and well known that for any $g\in G$, there is $n_j$ so that $g^{n_j}$ goes to the identity as $j\to\infty$.  Applied to the character group and $\chi_\fre$, we find $n_j$ so that $\chi_\fre^{n_j}\to 1$.  Thus $F_{n_j}\to F_1\equiv 1$, so by continuity of $\chi\mapsto\norm{F_\chi}$ and Remark~2 after Theorem~\ref{T1.1}, $W_{n_j}(\fre,x_0)$ converges to $2$.  Given the lower bound \eqref{3.11}, we get the result.
\end{proof}

\begin{theorem} \lb{T5.3} Let $\fre\subset\bbR$ be a compact PW set with DCT and let $x_0\in\bbR\bs\fre$.  Then $W_n(\fre,x_0)$ has a limit as $n\to\infty$ if and only if $\fre$ is a single interval and, in that case, the limit is $2$.
\end{theorem}

\begin{remark} For Chebyshev polynomials, this is a result of Totik \cite[Theorem~3]{Tot14}.  Indeed, he doesn't need the PW and DCT conditions. 
\end{remark}

\begin{proof} By Theorem~\ref{T5.2}, if the limit exists, it is $2$.  By Theorem~\ref{T1.1}, if $\chi$ is any limit point of $\chi_\fre^n$, we must have $\norm{F_\chi}=1$.  By the maximum principle, this can only happen if $\chi=1$.  It is easy to see that if $g$ is an element of a torus and the only limit point of $g^n$ is the identity, then $g$ is the identity.  Thus $\chi_\fre$ is $1$.  But the phase change of a simple closed curve in $\bbC\bs\fre$ enclosing a component of $\fre$ is the harmonic measure of the component within and if that is always $0$ or $1$, we have that $\fre$ has no gaps, i.e., is an interval.
\end{proof}

It is natural to ask if there is a similar universal result on the upper bound, that is, if the $\limsup$ always saturates \eqref{3.15}.  As explained in Example~\ref{E4.8}, if the orbit of $\chi_\fre$ is dense, then $\limsup$ saturates \eqref{3.15}.  However, for non-connected period-$n$ sets $\fre\subset\bbR$ there are values of $x_0\in\bbR\bs\fre$ such that the $\limsup$ does not saturate \eqref{3.15}.  Indeed, for such period-$n$ sets there are only finitely many limit point gap sets and the $\limsup$ saturates \eqref{3.15} only when one of the limit point gap sets contains the critical points of $g_\fre(\cdot,x_0)$ in each gap.  As we'll show below, the critical points of $g_\fre(\cdot,x_0)$ are not constant in $x_0\in I$ for any interval $I\subset\bbR\bs\fre$ and hence the $\limsup$ does not saturate \eqref{3.15} for $x_0$ in a dense subset of $\bbR\bs\fre$.

By contradiction, suppose $c$ is a critical point of $g_\fre(\dott,x_0)$ for all $x_0\in I$. Let $h(z)=\pd_t g_\fre(t,z)|_{t=c}$.  Since $g_\fre(t,\dott)=g_\fre(\dott,t)$ has a logarithmic pole at $t$, it follows that $h$ is a non-constant harmonic function on $\bbC\bs(\fre\cup\{c\})$. By assumption, $h\equiv 0$ on $I$ and hence $\pd_x h=0$ on $I$.  The symmetry $g_\fre(z,t)=g_\fre(\bar z,t)$ implies that
$$
\pd_y g_\fre(x+iy,t)|_{y=0}=0 \; \mbox{ for all } \; t,x\in\bbR\bs\fre, \; t\neq x.
$$
Therefore, $\pd_t\pd_y g_\fre(x+iy,t)|_{y=0}=0$ and so $\pd_y h=0$ on $\bbR\bs(\fre\cup\{c\})$.  In conclusion, we see that $\nabla h=0$ on an interval $I$ and since $h$ is harmonic, it follows that $h$ must be identically constant, a contradiction.

\bigskip\noindent{\bf Acknowledgments.}
We would like to thank M.\ Ismail, D.\  Lubinsky, and K.\ Schiefermayr for useful comments.


\end{document}